\documentclass[a4paper,11pt]{amsart}
\usepackage{amssymb, amsmath}
\usepackage{amscd}
\numberwithin{equation}{section}
\usepackage{epsfig}
\usepackage[all]{xy} 

\newcommand{\pp}{\mathbb{P}}

\newcommand{\C}{\mathbb{C}}
\newcommand{\cc}{\mathbb{C}}

\newcommand{\TT}{\mathbb{T}}

\newcommand{\Homom}{\mathrm{Hom}}
\newcommand{\Endom}{\mathrm{End}}
\newcommand{\Ker}{\mathrm{Ker}}

\newcommand{\rank}{\mathrm{rk}}
\newcommand{\Sym}{\mathrm{Sym}}

\newcommand{\cO}{\mathcal{O} }
\newcommand{\Kx}{K_{X}}
\newcommand{\sRat}{\underline{\mathrm{Rat}}}
\newcommand{\Rat}{\mathrm{Rat}}
\newcommand{\sPrin}{\mathrm{\underline{Prin}}}
\newcommand{\Prin}{\mathrm{Prin}}

\newcommand{\Sec}{\mathrm{Sec}}

\newcommand{\cE}{\mathcal{E}}
\newcommand{\cM}{\mathcal{M}}

\newcommand{\cQ}{\mathcal{Q}}
\newcommand{\cU}{\mathcal{U}}
\newcommand{\cW}{\mathcal{W}}
\newcommand{\tU}{\tilde{U}}
\newcommand{\Iden}{\mathrm{Id}}

\newcommand{\Spn}{\mathrm{Sp}_{n}}
\newcommand{\GL}{\mathrm{GL}}

\newcommand{\phoe}{\pp H^{1}(X, \Sym^{2}E)}
\newcommand{\hoe}{H^{1}(X, \Sym^{2}E)}

\newcommand{\sLag}{s_{\mathrm{Lag}}}
\newcommand{\Gpn}{\mathrm{Gp}_{n}}

\newtheorem{theorem}{{\textbf Theorem}}[section]
\newtheorem{proposition}[theorem]{{\textbf Proposition}}
\newtheorem{corollary}[theorem]{{\textbf Corollary}}
\newtheorem{lemma}[theorem]{{\textbf Lemma}}

\newtheorem{defn}[theorem]{{\textbf Definition}}

\newtheorem{remit}[theorem]{{\textbf Remark}}
\newenvironment{remark}{\begin{remit}\rm}{\end{remit}}
\newenvironment{definition}{\begin{defn}\rm}{\end{defn}}

\title[Lagrangian subbundles of symplectic bundles]{Lagrangian subbundles of symplectic bundles\\over a curve}

\author{Insong Choe and George H.\ Hitching}

\begin{document}

\begin{abstract}
A symplectic bundle over an algebraic curve has a natural invariant
$\sLag$ determined by the maximal degree of its Lagrangian
subbundles. This can be viewed as a generalization of the classical
Segre invariants of a vector bundle. We give a sharp upper
bound on $\sLag$ which is analogous to the Hirschowitz bound on the classical
Segre invariants. Furthermore, we study the stratifications induced by $\sLag$ on moduli
spaces of symplectic bundles, and get a full
picture for the case of rank four.
\end{abstract}



\maketitle

\section{Introduction}

Let $X$ be a smooth projective curve of genus $g \ge 2$ over $\cc$.
Let $G$ be a connected reductive algebraic group over $\cc$, and $P
\subset G$ a parabolic subgroup. Given a $G$-bundle $V$ over $X$, we
have an associated $G/P$-bundle $\pi \colon V/P \to X$. For a
section $\sigma$ of $\pi$, consider the normal bundle $N_\sigma$
over $\sigma(X) \cong X$ in $V/P$. We define
\[
s(V;P) := \min_\sigma \{ \deg N_\sigma \},
\]
where $\sigma$ runs through all sections of $V/P$.

It is easy to show that $s(V;P) > - \infty$; see Holla--Narasimhan \cite[Lemma
2.1]{HN}. A section $\sigma$ is called a {\it minimal section} of
$V$ if $\deg N_\sigma = s(V;P)$. According to Ramanathan \cite[Definition 2.13]{Rth}, $V$
is a semistable $G$-bundle if and only if $s(V;P) \ge 0$ for every
maximal parabolic subgroup $P$ of $G$. In general, it can be said
that the invariant $s(V;P)$ measures the degree of stability of $V$
with respect to $P$. The invariant $s(-,P)$ is a lower
semicontinuous function on any parameter space of $G$-bundles over
$X$, so defines a stratification on the moduli space of semistable
$G$-bundles over $X$.

The geometry of this stratification has been extensively studied in the case $G = \GL_n$. To review some of the results, first note
that a topologically trivial $\GL_n$-bundle $V$ is nothing but a
vector bundle of rank $n$ and degree zero, which we also denote $V$. For each $1 \le r \le n-1$, there is a maximal
parabolic subgroup $P_r$ of $\GL_n$ which is unique up to
conjugation.  A section $\sigma$ of the associated Grassmannian
bundle $V/P_r$ gives a rank $r$ subbundle $E$ of $V$, and vice versa. Since
$N_\sigma = \Homom(E, V/E)$, we get
\[
s(V;P_r) = \min \{ -n \deg E \ : \ E \subset V, \  \text{rk}(E) = r
\}.
\]
This coincides with the well-known invariant $s_r(V)$ of a vector
bundle $V$, sometimes called the \textit{$r$-th Segre invariant}. In
the literature, a subbundle $E$ is called a {\it maximal subbundle}
of $V$ if $s_r(V) = -n \deg E$.

The first result on the invariant $s_r$ is the upper bound
\[
s_r(V) \le g \cdot r(n-r).
\]
This was obtained by Nagata \cite{Nag} for $n=2$ and by Mukai and
Sakai \cite{MS} in general. Later, Hirschowitz obtained the
following sharp bound \cite[Th\'{e}or\`{e}me 4.4]{Hir} (see also \cite{CH}):

\begin{proposition} \label{Hir} For a bundle $V$ as above, we have
\[
s_r(V) \le r(n-r)(g-1) + \varepsilon,
\]
where $0 \le \varepsilon <n$ and $r(n-r)(g-1) + \varepsilon \equiv 0
\mod n$. \qed
\end{proposition}
Next, let us recall the results on the stratification defined by
$s_r(V)$. Let $\cU(n)$ be the moduli space of semistable bundles
over $X$ of rank $n$ and degree zero. For each integer $s$ divisible
by $n$, we define
\[
\cU(n;r, s) := \{ V \in \cU(n) \ : \ s_r(V) \le s \}.
\]
By Proposition \ref{Hir}, we have $\cU(n;r,s) = \cU(n)$ if $s \ge
r(n-r)(g-1)$.
\begin{proposition}
{\rm (Brambila-Paz--Lange \cite{BPL}, Russo--Teixidor i Bigas \cite{RT})} \ \ For each integer $k$ with $0 < k
\le k_0 := \left\lfloor \frac{r(n-r)(g-1)}{n} \right\rfloor$, $\cU(n; r,kn)$ is an irreducible closed subvariety of $\cU(n)$. Also,
\[
{\rm codim} \; \cU(n; r,kn) = r(n-r)(g-1) - kn
\]
for each $k \leq k_0$. \qed
\end{proposition}
\noindent Moreover, each variety $\cU(n;r,kn)$ can be described by
using the extension spaces of fixed type, see \cite[Theorem
0.1]{RT}.

On the other hand, not much seems to have been studied on the
properties of the invariant $s(V;P)$ of a $G$-bundle $V$ in general, except the
following universal upper bound on $s(V;P)$ obtained by Holla and
Narasimhan \cite[Theorem 1.1]{HN}:
\begin{proposition} Fix a parabolic subgroup $P$ of
$G$. For every $G$-bundle $V$, we have
\begin{equation} \label{HN}
s(V;P) \le g \cdot \dim(G/P). \qed
\end{equation}
\end{proposition}
Note that when $G = \GL_n$, this coincides with the Mukai--Sakai bound discussed above, which is not sharp.
As in the vector bundle case, the invariant $s(V;P)$ induces a stratification on the moduli space of semistable principal $G$-bundles over $X$.

In this paper, we study the geometry of the stratification when $G = \Spn$. A vector
bundle $W$ is \textit{symplectic} if there exists a nondegenerate
bilinear alternating form $\omega\colon W \otimes W \to \cO_X$. Such
a bundle always has even rank $2n$. It is easy to see that a
vector bundle of rank $2n$ is symplectic if and only
if the associated principal $\mathrm{\GL}_{2n}$-bundle admits a
reduction of structure group to $\Spn$.
A subbundle $E$ of $W$ is called \textit{isotropic} if $\omega|_{E
\otimes E} = 0$. By linear algebra, the rank of an isotropic
subbundle is at most $n$, and an isotropic subbundle of rank $n$ is
called a \textit{Lagrangian} subbundle. A Lagrangian
subbundle of $W$ corresponds to a section of the associated $\Spn /
P$-bundle $W/P$, where $P$ is the maximal parabolic subgroup of
$\Spn$ preserving a fixed Lagrangian subspace of $\cc^{2n}$. Let us
abbreviate $s(W;P)$ to $\sLag(W)$, where the lower case indicates
``Lagrangian''. Since $\Spn/P$ is none other than the Lagrangian
Grassmannian, the normal bundle of a section is given by $\Sym^2
E^*$. Recall that for any vector bundle $V$, we have
\[ \deg(\Sym^{2}V) = (\rank(V) + 1)\deg(V), \]
so
\[
\sLag(W) = \min \{ -(n+1) \deg E \ : \ E\ \text{a Lagrangian
subbundle of } W \}.
\]
According to the bound (\ref{HN}), we get
\begin{equation} \label{HN-bound}
\sLag(W) \le \frac{n(n+1)}{2} g.
\end{equation}
In other words, every symplectic bundle $W$ admits a Lagrangian
subbundle of degree at least $-\frac{n}{2} g$.

One may compare this with the Hirschowitz bound on the $n$-th Segre
invariant $s_n$, which says that as a vector bundle, $W$ admits a subbundle of half rank with
degree at least $ - \left\lceil \frac{n}{2} (g-1) \right\rceil$. We prove that
this slightly nicer bound is still valid for symplectic bundles.
\begin{theorem} \label{thm1}
For every symplectic bundle $W$ of rank $2n$, we have
\[
\sLag(W) \le \frac{1}{2} \left( n(n+1)(g-1) + (n+1) \varepsilon \right),
\]
where $0 \le \varepsilon  \le 1$ is such that $n(g-1) + \varepsilon$ is
even. This bound is sharp in the sense that the equality holds for a
general bundle $W$ in the moduli space $\cM_{2n}$ of semistable
symplectic bundles of rank $2n$ over $X$.  \qed
\end{theorem}

Next, consider the stratification on $\cM_{2n}$ given by the
invariant $\sLag$. For each $k
>0$, let
\[
\cM_{2n}^k := \{ W \in \cM_{2n} \ : \ \sLag(W) \le (n+1)k \}.
\]
By semicontinuity, $\cM_{2n}^k$ is a closed subvariety of $\cM_{2n}$
and $\sLag$ induces a stratification on $\cM_{2n}$. In particular
when $n=1$, since $\mathrm{Sp}_1$ is isomorphic to $\mathrm{SL}_2$,
this reduces to the stratification already studied on the moduli space
$\mathcal{SU}(2, \cO_{X})$ of semistable bundles of rank two with
trivial determinant. We prove the following result on the
stratification on the moduli space $\cM_4$ of semistable rank four
symplectic bundles over $X$. Note that $\cM_4^{g-1}$ is the whole
space $\cM_4$ by Theorem \ref{thm1}.

\begin{theorem} \label{mainthm}
For each $e$ with $1 \le e \le g-1$, the locus $\cM_4^{e}$ is an irreducible closed subvariety of $\cM_4$ of dimension
$7(g-1)+3e$. \qed
\end{theorem}

In the case of genus two, this was proven in \cite{Hit3}.
The key ingredient of the proof there was a symplectic version of Lange and
Narasimhan's description \cite{LN} of the Segre invariant using
secant  varieties (also see \cite{CH} for a higher rank version in
the case of vector bundles). In this paper, we generalize the method
and results of \cite{Hit3} to the case of arbitrary genus. We will
see an interesting variation of Lange and Narasimhan's picture in
the case of symplectic bundles: a relation between the invariant
$\sLag$ and the higher secant varieties of certain fiber bundles over
$X$.

This paper is organized as follows. In \S2, we provide the basic
setup for our discussion. In particular, a relation between the
higher secant varieties and the Segre invariants will be established
in Theorem \ref{seginvsec}. In \S 3, we prove Theorem \ref{thm1}
using this relation combined with the Terracini Lemma. In \S4, we study
symplectic bundles of rank four in more detail. We will see that the
relation discussed in \S2 can be improved to yield a nice picture in
this case (Theorem \ref{segsec}). This enables us to prove Theorem
\ref{mainthm}. Finally a remark will be given on the comparison
between the two stratifications defined by $s_2$ and $\sLag$.

To streamline the
arguments, the proofs of two propositions will be postponed to \S5; namely,
the irreducibility of certain spaces of principal parts, and a variant of Hirschowitz' lemma.

Throughout this paper, we work over the field $\cc$ of complex
numbers.

\section{Symplectic extensions and lifting criteria}

In this section, we establish basic results on symplectic
extensions for our further discussions.

\subsection{Symplectic extensions}

Let $W \to X$ be a symplectic bundle and $E \subset W$ a
subbundle. There is a natural short exact sequence
\[ 0 \to E^{\perp} \to W \to E^{*} \to 0, \]
where $E^\perp$ is the orthogonal complement of $E$ with respect to
the symplectic form on $W$. If $E$ is a Lagrangian subbundle of $W$,
then $E = E^{\perp}$ and we get a class
\[
\delta(W) \in H^1 (X, \Homom(E^{*}, E)) = H^1(X, E \otimes E).
\]
Note that we have the decomposition
\[
H^1 (X, E \otimes  E) \ \cong \ \hoe \oplus H^1(X, \wedge^2 E).
\]
\begin{lemma}
Assume that $E$ is simple. An extension $0 \to E \to W \to E^{*} \to
0$ has a symplectic structure with respect to which $E$ is isotropic
if and only if the class $\delta(W)$ belongs to $\hoe$.
\label{fundam} \end{lemma}
\begin{proof}
Due to S.\ Ramanan; see \cite[$\S$2]{Hit2} for a proof.
\end{proof}

\subsection{Cohomological criterion for lifting}

Henceforth, an extension $0 \to E \to W \to E^{*} \to 0$ will be called \textsl{symplectic} if $W$ admits a symplectic structure with respect to which the subbundle $E$ is Lagrangian. By Lemma \ref{fundam}, this is equivalent to the condition $\delta(W) \in H^{1}(X, \Sym^{2}E)$.

We recall the notion of a \emph{bundle-valued principal part} (see Kempf
\cite{Kem} for corresponding results on line bundles). For any bundle $V$ over $X$, we have an exact sequence of $\cO_{X}$-modules
\begin{equation} \label{ratprin}
0 \to V \to \sRat(V) \to \sPrin(V) \to 0
\end{equation}
where $\sRat(V)$ is the sheaf of rational sections of $V$ and
$\sPrin(V)$ the sheaf of principal parts with values in $V$. We
denote their groups of global sections by $\Rat(V)$ and $\Prin(V)$
respectively. The sheaves $\sRat(V)$ and $\sPrin(V)$ are flasque, so
we have the cohomology sequence
\begin{equation} \label{cohomseq}
0 \to H^{0}(X, V) \to \Rat(V) \to \Prin(V) \to H^{1}(X, V) \to 0.
\end{equation}
We denote $\overline{s}$ the principal part of $s \in \Rat(V)$, and
we write $[p]$ for the class in $H^{1}(X, V)$ of $p \in \Prin(V)$.
Any extension class in $H^{1}(X, V)$ is of the form $[p]$ for some
$p \in \Prin(V)$, which is far from unique in general.

Now consider an extension of vector bundles
\begin{equation} \label{stdext}
0 \to E \to W \to E^* \to 0
\end{equation}
and an elementary transformation $F$ of $E^*$ defined by the
sequence
\[
0 \to F \stackrel{\mu}{\to} E^* \to \tau \to 0
\]
for some torsion sheaf $\tau$. We say that $F$ \textit{lifts to} $W$
if there is a sheaf injection $F \to W$ such that the composition $F
\to W \to E^*$ coincides with the elementary transformation $\mu$.
We quote two results from \cite{Hit2}:
\begin{lemma} \cite[Corollary 3.5 and Criterion 3.6]{Hit2} \label{cohomlift}
\ \ Suppose that $h^{0}(X, \Homom(E^{*}, E)) = 0$ and $E$ is simple. Let $W$ be an
extension of type (\ref{stdext}) with class $\delta(W) \in H^{1}(X, E \otimes E)$.
\begin{enumerate}
\item There is a one-to-one correspondence between principal parts $p \in \Prin(E \otimes E)$ such that
$[p] = \delta(W)$, and elementary transformations $F$ of $E^*$ lifting
to $W$ as a subbundle, given by \ $p \longleftrightarrow \Ker \left(
p \colon E^{*} \to \sPrin(E) \right)$.
\item The subbundle corresponding to $\Ker(p)$ is isotropic in $W$ if and only if $^tp = p$. \qed
\end{enumerate}
\end{lemma}

\subsection{Subvarieties of the extension spaces} \label{alternative}

Let $V$ be a vector bundle with $h^{1}(X, V) \neq 0$. We describe a
rational map of the scroll $\pp V$ into the projective space $\pp
H^{1}(X,V)$. Let $\pi \colon \pp V \to X$ be the projection. We have
the following sequence of identifications:
\begin{eqnarray*}
H^{1}(X,V) &\cong& H^{0}(X, \Kx \otimes V^{*} )^{\vee}  \\
 &\cong& H^{0}(X, \Kx \otimes \pi_{*}\cO_{\pp V}(1) )^{\vee} \\
 &\cong& H^{0}(X, \pi_{*} \left( \pi^{*}\Kx \otimes \cO_{\pp V}(1) \right) )^{\vee}  \\
 &\cong& H^{0}(\pp V, \pi^{*}\Kx \otimes \cO_{\pp V}(1) )^{\vee}.
 \end{eqnarray*}
Hence via the linear system $|\pi^{*}\Kx \otimes \cO_{\pp V}(1)|$,
we have a map $\varphi \colon \pp V \dashrightarrow \pp H^{1}(X,
V)$. For $V = \Homom(E^{*}, E) = E \otimes E$, we get a map
\[
\varphi: \pp (E \otimes E) \dashrightarrow \pp H^1(X, E \otimes E).
\]
Also for $V = \Sym^2 E$, we get a map
\[
\pp \Sym^2 E \dashrightarrow \pp  \hoe,
\]
which is nothing but the restriction of $\varphi$ to $\pp \Sym^2 E$.
If $v \in V$ is nonzero, we will abuse notation and let $v$ also denote the point of $\pp E$
corresponding to the vector $v$.
By restricting $\varphi$ further to the subvariety $\pp E \subset
\pp \Sym^2 E$ embedded by the map $v \mapsto v \otimes v$, we
get
\[
\psi: \pp E \dashrightarrow \pp \hoe.
\]

Now we give another way to define $\varphi$ and $\psi$, which will be
convenient in what follows. This idea was used by Kempf and Schreyer
\cite[$\S$1]{KS} to define the canonical map $X \to |\Kx|^{\vee}$, and
generalized in \cite[$\S$3.3]{Hit3}.

For each $x \in X$, we have an exact sequence of sheaves
\[ 0 \to V \to V(x) \to \frac{V(x)}{V} \to 0, \]
which is a subsequence of (\ref{ratprin}). Taking global sections,
we obtain
\begin{equation} \label{globalsections}
0 \to H^{0}(X,V) \to H^{0}(X, V(x)) \to V(x)|_{x} \stackrel{\partial}{\to} H^{1}(X, V)
\to \cdots
\end{equation}
The restriction of $\varphi$ to the fiber $\pp V|_x$ is identified
with the projectivized coboundary map $\partial$ in
(\ref{globalsections}). In view of (\ref{cohomseq}), it can be
identified with the map taking $v \in \pp V$ to the cohomology class
of a $V$-valued principal part supported at $x$ with a simple pole
in the direction of $v$. Similarly, $\psi$ can be identified with the map taking $v \in E|_x$ to the cohomology class of a $\Sym^{2}E$-valued principal part supported at $x$ with a simple pole along $v \otimes v$.

\begin{lemma} \label{psiemb}
Consider the above rational maps $\varphi$ and $\psi$ for a vector
bundle $E$.
\begin{enumerate}
\item The map $\varphi$ is an
embedding if $E$ is stable and $\mu(E) \le -1$.
\item If $E$ is stable with $\mu(E) \le -1/2$ then $\varphi$ is base-point
free.
\item If $g \ge 3$ and $E$ is general in $\cU(2,d)$, $d<0$, then $\psi$ is
an embedding.
\end{enumerate}
\end{lemma}
\begin{proof} \ (1) For any divisor $D$ on $X$, write $V(D)$ for the bundle $V \otimes
\cO_{X}(D)$. One can easily see that $\varphi$ is base point free if $H^{0}(X, V(x)) = 0$ for all $x \in X$, and an embedding if $H^{0}(X, V(D)) = 0$ for all effective divisors $D$ of
degree 2. Since $E$ is stable, so is $E^*$, and the condition
$\mu(E) \le -1$ implies that
\[ \mu(E^{*}) \ge \mu(E) + 2. \]
If there were a nonzero map $\alpha \colon E^{*} \to E(D)$ for some
$D \in \Sym^{2}X$, then we would have
\[
\mu(E) + 2 \ \le \ \mu(E^*) \ < \ \mu(\mathrm{Im} \alpha) \ < \ \mu
(E(D)) = \mu(E) + 2,
\]
which is a contradiction. Thus $H^{0}(X, \Homom(E^{*}, E(D))) = 0$.

(2) The same argument as above shows that if $E$ is stable and
$\mu(E) \le -1/2$, then $ H^{0}(X, V(x)) = 0$ for
all $x \in X$. Therefore, in this case $\varphi$ is base-point free.

(3) By (1), $\varphi$ is an embedding for $d < -1$. Now suppose $d =
-1$. By (2), we know that $\varphi$ is base-point free, and hence so is $\psi$.

To show that $\psi$ is an embedding, it suffices to show that for
any $x,y \in X$,
\begin{itemize}
\item No two principal parts of the form
\[ \frac{e \otimes e}{z} \quad \hbox{and} \quad \frac{f \otimes f}{w} \]
are identified in $H^{1}(X, \Sym^{2} E)$ for any nonzero $e \in
E|_x$ and $f \in E|_y$, where $z$ and $w$ are local coordinates at
$x$ and $y$ respectively, and
\item the principal part
\[ \frac{e \otimes e}{z^2} \]
is not cohomologically zero for any nonzero $e \in E|_x$.
\end{itemize}
\indent The first statement is equivalent to saying that there does not
exist a symmetric map
\[ \alpha \colon E^{*} \to E(x + y) \]
with principal part equal to
\[
\frac{e \otimes e}{z} - \frac{f \otimes f}{w}. \label{ppt}
\]
\indent First assume that $\alpha$ is generically injective. Note
that any such $\alpha$ has rank one at $x$ and $y$. Hence the
non-zero section $\det(\alpha)$ of the line bundle
\[ \Homom (\det E^* ,\det E (2x+2y)) \cong (\det E)^{2}(2x+2y) \]
vanishes at $x$ and $y$. Thus $(\det E)^{2}(x+y)$ is trivial, so $(\det E)^{-2}$ is effective. But if $g \geq 3$, a general $E \in \cU(2,-1)$ does not
have this property.

Next, assume that $\alpha $ is of generic rank 1. Then $\alpha$
factorizes as
\[ E^{*} \to M \to E(x+y) \]
for some rank 1 sheaf $M$ of degree $m$. The surjection $E^* \to M$
and the injection $M \to E(x+y)$ imply that $E$ admits a line
subbundle of degree at least
\[
\max \{ \deg M^*, \deg M(-x-y) \} \ = \ \max \{ -m, m-2 \} \ \ge -1.
\]
This shows that $E$ is not general in $\cU(2,-1)$ for $g \ge 3$ by
a known property of the Segre invariant of $E$, which is analogous to
Proposition \ref{Hir}: For a rank 2 bundle $E$ of degree $d$, define
\[
s_1(E) : = \min \{ d - 2 \deg L \}
\]
where $L$ runs through all the line subbundles of $E$. Then
by Lange--Narasimhan \cite[Proposition 3.1]{LN}, we have
\begin{equation} \label{s_1}
g-1 \le s_1(E) \le g
\end{equation}
for a general $E \in \cU(2,d)$. This in particular shows that for $g
\ge 3$, a general $E \in \cU(2,-1)$ does not admit a line subbundle
of degree $\ge -1$.

Hence we conclude that there is no such symmetric map $\alpha$.
This proves the first statement. The second statement can be proven by a
similar argument with $x=y$.
\end{proof}

\subsection{Geometric criterion for lifting}

Firstly we adapt some ideas from \cite[$\S$3]{CH} to the present
situation.

\begin{definition} \label{genelm} Let $V \to X$ be a vector bundle. Then an elementary transformation
\[ 0 \to \tilde{V} \to V \to \tau \to 0 \]
is called \emph{general} if $\tau \cong \bigoplus_{i=1}^{k}
\C_{x_i}$ for distinct points $x_{1}, \ldots , x_k$ of $X$.
\end{definition}

The word ``general'' is justified as follows. As in
\cite[$\S$3]{CH}, one can consider a parameter space $\cQ_k(V)$ of
elementary transformations of $V$ of degree $\deg(V) - k$ and
observe that the locus of general elementary transformations is open
and dense in $\cQ_k(V)$.

\begin{definition} Let $Y$ be a closed subvariety of a projective space $\pp^N$.
The \emph{$k$-th secant variety} of $Y$, denoted by $\Sec^k Y$, is
the Zariski closure of the union of all the $k$-secants to $Y$, that
is, all the projective subspaces of $\pp^N$ spanned by $k$ distinct
points of $Y$.
\end{definition}

\begin{definition} For each $x \in X$, we denote $\Delta|_x$ the projectivization of the set of all rank one linear maps $E^{*}|_{x} \to E|_x$. The union of all the $\Delta|_x$ is a fiber bundle $\Delta$ inside the scroll $\pp ( E \otimes E )$, which we call the \emph{decomposable locus}. We will also use the fiber subbundle of $\Delta$ given as the image of the Veronese embedding $\pp E \hookrightarrow \pp ( E \otimes E )$. \end{definition}



Now we will consider a bundle $E \to X$ of rank $n$ which satisfies either condition
(1) or the condition (3) in Lemma \ref{psiemb} so that either
\[ \varphi \colon \pp(E \otimes E) \to \pp H^1(X, E \otimes E) \quad \hbox{or} \quad \psi: \pp E \to \pp \hoe \]
is an embedding. Consider the following diagram:
\[
\begin{CD}
\Delta @>>> \pp(E \otimes E) @>>> \pp H^1(X, E \otimes E)  \\
@AAA  @AAA  @AAA  \\
\pp E @>>> \pp (\Sym^2E) @>>> \pp H^1(X, \Sym^2 E).
\end{CD}
\]
If $\psi$ is an embedding, the composition of the two maps in the bottom
row is an inclusion, whereas if $\varphi$ is an embedding, all the arrows are inclusions.
\begin{lemma} \label{genlift}
Consider a bundle $W$ fitting into a nontrivial symplectic extension
of $E^*$ by $E$ with class $\delta(W) \in \phoe$.
\begin{enumerate}
\item If $W$ admits an isotropic
lifting of an elementary transformation $F$ of $E^{*}$ with
$\deg(E^{*}/F) \leq k$, then $\delta(W)$ belongs to $\Sec^k \pp E$.
\item If the class $\delta(W)$ corresponds to a general point of $\Sec^k \pp E$,
then $W$ admits an isotropic lifting of a general elementary
transformation $F$ of $E^{*}$ with $\deg(E^{*}/F) \leq k$.
\end{enumerate}
\end{lemma}
\begin{remark} \label{ordinarylifting}
An analogous lifting condition for the map $\varphi$ was
considered in \cite[Theorem 4.4]{CH}, where a criterion on
the lifting of an elementary transformation of $E^*$ without the
isotropy condition, was given in terms of the higher secant variety of the scroll $\Delta
\subset \pp H^1(X, E \otimes E)$ instead of $\pp E \subset \pp H^1(X, \Sym^2 E)$.
\end{remark}
\begin{proof} Basically we follow the same idea underlying the proof of
\cite[Theorem 4.4]{CH}, except that we use the language of principal
parts here.

By Lemma \ref{cohomlift} (2), an elementary transformation $F$ of
$E^*$ lifts to $W$ isotropically if and only if the class
$\delta(W) = [p]$ for some symmetric principal part $p$
such that
\[ F \subseteq \Ker \left( p \colon E^{*} \to \sPrin(E) \right). \]
First assume that $F$ is a general elementary transformation so that
$p$ is a linear combination of $k$ symmetric principal parts $p_{1},
\ldots , p_{k}$ supported at distinct points $x_{1}, \ldots , x_k$,
with
\begin{equation}
p_{i} \in \Delta|_{x_i} \subseteq (E \otimes E)|_{x_i} \cong
\frac{\left( E \otimes E \right)(x_{i})}{(E \otimes E)|_{x_i}}
\label{elemppt}
\end{equation}
for each $i$. By symmetry, each $p_i$ belongs to $\Delta|_{x_i} \cap \pp \Sym^{2}E|_{x_i} = \pp E|_{x_i}$. By our alternative definition of $\varphi \colon \pp (E \otimes E)
\rightarrow \pp H^{1}(X, E \otimes E)$ at the end of
\S\ref{alternative}, the point $\delta(W) = [p]$ lies on the secant
plane spanned by $k$ distinct points of $\pp E$, in other words,
$\delta(W) \in \Sec^k(\pp E)$.

The proof of (1) will be completed by taking limit, once we know the irreducibility
of the space of  symmetric principal parts defining elementary transformations of degree $\deg E^{*} - k$. The detailed proof of
this rather obvious-looking fact will be given in \S5.1.

To show (2), suppose $\delta(W)$ lies on a secant plane spanned by
$k$ general points of $\pp E$. This means that $\delta(W)$ can be
defined by a linear combination $p = \sum \lambda_i p_i$ of at most
$k$ symmetric principal parts $p_i$ supported at distinct points, where
$p_i \in \pp E|_{x_i}$. By Lemma \ref{cohomlift} (2), the kernel of
$p \colon E^{*} \to \sPrin(E)$ lifts to $W$ isotropically. Also,
$\Ker(p)$ is a general elementary transformation of $E^*$ whose
degree is at least $\deg E^* -k$.
\end{proof}

From the above discussion, we obtain:
\begin{theorem} Let $E$ and $W$ be as above.
If $\delta(W) \in \Sec^{k}\pp E$, then we have $\sLag(W) \leq (n+1)(k + \deg E)$. \label{seginvsec} \end{theorem}
\begin{proof}
By Lemma \ref{genlift} (2), if $\delta(W)$ is a general point of
$\Sec^k \pp E$, then there is some elementary transformation $F
\subset E^*$ lifting to $W$ isotropically such that $\deg(F) \ge
\deg E^* -k$. Hence by definition, $\sLag(W) \le (n+1)(k+ \deg E)$. By
the semicontinuity of the invariant $\sLag$, this inequality still
holds for any $W$ with $\delta(W) \in \Sec^k \pp E$.
\end{proof}

\begin{remark} \label{converse}
(1) In \cite[Theorem 4.4]{CH}, it was proven that if $\delta(W) \in
\Sec^{k}\Delta$ then we have $s_{n}(W) \leq 2n(k+ \deg E)$, that is, $W$ contains a rank $n$ subsheaf of degree $\deg(E^{*}) - k$, which is not necessarily isotropic. We will return to this phenomenon in \S \ref{difference}.

(2) One can ask if the converse holds in Theorem \ref{seginvsec}.
This is a subtle question. Certainly the condition $\sLag(W) \le
(n + 1)(k + \deg E)$ implies that $W$ admits an isotropic subbundle $F$ of
rank $n$ and degree $\ge \deg E^* -k$. But in general $F$ need not
come from an elementary transformation of $E^*$, due to the possible
existence of the diagram of the form
\begin{equation} \label{exceptional}
\begin{CD}
0 @>>> E @>>> W @>>> E^{*} @>>> 0 \\
@. @AAA  @AAA  @AAA  @.\\
0 @>>> G @>>> F @>>> H @>>> 0
\end{CD}
\end{equation}
where $\rank(G) \geq 1$. In other words, for $n > 1$, it can happen that a maximal Lagrangian
subbundle of $W$ may intersect $E$ in a nonzero subsheaf. But in
\S4.1, we will show that this kind of diagram appears only in a
restricted way for $n=2$.
\end{remark}

\section{Upper bound on $\sLag$} \label{Hirschowitz}

One can easily see that the invariant $\sLag$ has no lower bound by
considering, for instance, a decomposable bundle $E \oplus E^*$
where $E$ is a bundle of arbitrarily large degree.  In this section,
we prove Theorem \ref{thm1} which gives us the sharp upper bound on
$\sLag$.

In \cite{CH}, the Hirschowitz bound in Proposition \ref{Hir} was
reproved using the relation between the invariant $s_r$ and the
geometry of the higher secant varieties of the ruled varieties in
the extension spaces. We would like to adapt the same idea to the
case of symplectic bundles by applying the results in the previous
section. We begin with the following observation.

\begin{lemma}
A general symplectic bundle  $W \in \cM_{2n}$ has a Lagrangian
subbundle $E$ of degree $\le -\frac{1 }{2}n(g-1)$ which is general
as a vector bundle. \label{structure}
\end{lemma}
\begin{proof}
First we show that every symplectic bundle $W$ has a Lagrangian
subbundle. Let $U \subset X$ be an open set over which $W$ is
trivial. Any Lagrangian subspace of $\cc^{2n}$ yields a Lagrangian
subbundle of $W|_U$. This corresponds to a rational reduction of
structure group of the associated $\Spn$-bundle of $W$ to the
maximal parabolic subgroup preserving a fixed  Lagrangian subspace.
Since $X$ has dimension one, this can be extended to a reduction of
structure group over the whole of $X$, giving a Lagrangian subbundle
of $W$ (cf.\ Hartshorne \cite[Proposition I.6.8]{Har}).

Now let $W_0$ be a general symplectic bundle. Since $W_0$ admits a
Lagrangian subbundle, say $E_0$, we have an extension
\[ 0 \to E_{0} \to W_{0} \to E_{0}^{*} \to 0 \]
with class $\delta(W_{0}) \in H^{1}(X, \Sym^{2}E_{0})$. Now it is
well known that there is a versal deformation $\cE \to S \times X$
of $E_0$ parameterized by an irreducible variety $S$, whose general
member is a general stable bundle. Consider the direct image $R^1
q_* (\Sym^2 \cE)$, where $q \colon S \times X \to S$ is the
projection. This sheaf is locally free over the open subset
$S^\circ$ of $S$ consisting of the bundles $E$ with $h^0(X, \Sym^2
E) = 0$. We claim that $E_0$ belongs to $S^\circ$: Indeed, a nonzero
map $E_0^* \to E_0$ would induce an endomorphism of $W$ given
by composition $W_0 \to E_0^* \to E_0 \to W_0$. This is clearly
nonzero and nilpotent. But this contradicts the stability of $W$.


The associated projective bundle $\pp (R^1 q_* (\Sym^2
\cE)|_{S^{\circ}})$ over $S^\circ$ will be denoted $\pi \colon P
\to S^\circ$. By Lange \cite[Corollary 4.5]{Lan}, there is an exact
sequence of bundles over $P \times X$:
\[ 0 \to (\pi \times \Iden_{X})^{*}\cE \otimes p^{*}\cO_{P}(1) \to \cW \to (\pi \times \Iden_{X})^{*}\cE^{*} \to 0, \]
where $p \colon P \times X \to P$ is the projection, with the
property that for each $\delta$ in the fiber $\pi^{-1}([E])$ in
$P$, the restriction $\cW|_{\{ \delta \} \times X}$ is isomorphic to
the symplectic extension of $E^*$ by $E$ defined by $\delta$.
From the construction of $S^\circ$, a general point of $P$ corresponds to a
symplectic bundle admitting a Lagrangian subbundle which is
general as a vector bundle.

Finally we check that $\deg(E) =: -e \leq -\frac{1}{2} n(g-1)$. From
the above argument, the dimension of the space of symplectic bundles
admitting a Lagrangian subbundle of degree $-e$ is bounded above by
\[
\dim \cU(n,-e) + h^{1}(X, \Sym^{2}E) - 1 \ \ = \ \  n^{2}(g-1)+1 +
(n+1)e + \frac{1}{2}n(n+1)(g-1)  - 1.
\]
Since this should be at least $ \dim \cM_{2n} = n(2n+1)(g-1)$, we
obtain $e \ge \frac{n(g-1)}{2}$.
\end{proof}

To prove Theorem \ref{thm1}, it suffices to show that a general $W
\in \cM_{2n}$ has a Lagrangian subbundle of degree at least $-
\left\lceil \frac{1}{2}n(g-1) \right\rceil$. By Lemma
\ref{structure}, we know that $W$ has a Lagrangian subbundle $E$ which is
general in $\cU(n, -e)$.

By Lemma \ref{psiemb} (1), we may assume the map $\varphi \colon \pp(E \otimes E)
\to \pp H^{1}(X, E \otimes E)$ is an embedding for sufficiently large $e$. Consider its
restriction $\psi \colon \pp E \to \pp := \pp H^1(X, \Sym^2 E)$. By
Theorem \ref{seginvsec}, if the class $\delta(W) \in \pp$ belongs to the subvariety $\Sec^{k} \pp E$, then $W$ has
a Lagrangian subbundle of degree at least $e-k$. For the moment,
suppose that $\pp E$ has no secant defect in $\pp $ so that for each
$k \ge 1$, we have
\begin{eqnarray*}
\dim \Sec^k \pp E &=& \min \{ k  \dim \pp E + k  - 1, \ \dim \pp \}\\
&=& \min \{ k(n+1), (n+1)e + \frac{1}{2}n(n+1)(g-1) \} \ -1.
\end{eqnarray*}
By straightforward computation, $\Sec^k \pp E = \pp $ if $k \ge k_0
:= e + \left\lceil \frac{ 1 }{2} n(g-1) \right\rceil.$ This implies that $W$ has
a Lagrangian subbundle of degree at least $-\left\lceil\frac{ 1 }{2}
n(g-1)\right\rceil$, as was desired.

Hence the proof will be completed once we show the following.

\begin{proposition} \label{defect}
For a general bundle $E \in \cU(n, -e)$, $ e \gg 0$, the subvariety
$\pp E$ of $\pp := \pp H^1(X, \Sym^2 E)$ has no secant defect. \qed
\end{proposition}
The rest of this section is devoted to the proof of Proposition
\ref{defect}. First we recall the Terracini Lemma \cite{Ter}:
\begin{lemma} Let $Z \subset \pp^N$ be a projective variety and
let $z_{1}, \ldots, z_k$ be general points of $Z$. Then $\dim
\Sec^{k} Z  = \dim \langle \TT_{z_1} Z, \ldots, \TT_{z_k} Z
\rangle$, where $\TT_{z_i}Z$ denotes the embedded tangent space to
$Z$ in $\pp^N$ at $z_i$. \qed
\end{lemma}

To apply the Terracini Lemma,  let us find a description of the
embedded tangent spaces of $\pp E \subset \pp$. Let $v$ be a nonzero
vector of $E|_x$, and consider the elementary transformation $0 \to
E \to \hat{E} \to \C_{x} \to 0$ such that the kernel of $E|_{x} \to
\hat{E}|_x$ is spanned by $v$. Consider the induced elementary
transformation
\[
0 \to \Sym^{2}E \to \Sym^{2} \hat{E} \to \tau \to 0
\]
where $\tau$ is a torsion sheaf of degree $(n+1)$.

\begin{lemma} \label{tangent} For a general bundle $E \in \cU(n, -e)$, with $ e \gg 0$,
the embedded tangent space to $\pp E$ at $v$ in $\pp$ is given by
\[ \TT_{v} (\pp E) \ = \ \pp \Ker \left( H^{1}(X, \Sym^{2}E) \to H^{1}
( X, \Sym^{2} \hat{E} ) \right). \]
\end{lemma}
\begin{proof} Recall that $\Delta$ is the subvariety of $\pp (E \otimes E) \subset \pp H^1(X, E \otimes E)$ consisting of decomposable vectors. We have the commutative diagram
\begin{equation} \label{fandg}
\begin{CD}
 H^1(X, E \otimes E) @>f>>  H^{1}(X, \hat{E} \otimes \hat{E})  @>>> 0 \\
  @AAA  @AAA  @.\\
 H^{1}(X, \Sym^{2}E) @>g>>  H^{1} ( X, \Sym^{2} \hat{E} )  @>>> 0,
\end{CD}
\end{equation}
where the vertical arrows are inclusions. By \cite[Lemma 5.3]{CH},
the embedded tangent space of $\Delta$ at $v \otimes v$ in $\pp H^1(X, E \otimes E)$ is given by
$\TT_{v \otimes v} \Delta = \pp (\Ker f)$. From the inclusion $\pp E
\subset \Delta$ and the above diagram (\ref{fandg}), we get
\[ \TT_{v} \pp E \subseteq \TT_{v \otimes v} \Delta \ \cap \ \pp H^{1}(X,\Sym^{2}E) = \pp \Ker (f) \ \cap \ \pp H^{1}(X,\Sym^{2}E) = \pp \Ker (g). \]
Thus it suffices to show that $\pp \Ker (g)$ has
dimension $n = \dim \pp E $. Since $E$ is general, we may assume $\hat{E}$ is
stable. Then $\Sym^{2} E$ and $\Sym^{2}\hat{E}$ are semistable of
negative degree, so they have no nonzero sections. Thus $\Ker(g)$ has
dimension
\[
h^{1}(X, \Sym^{2}E ) - h^{1}(X, \Sym^{2}\hat{E}) \ = \ \deg
(\Sym^{2}\hat{E}) - \deg (\Sym^{2} E) \ = \ n+1,
\]
as desired. \end{proof}

\noindent \textit{Proof of Proposition \ref{defect}}. \ Let $F$ be
the general elementary transformation of $E$ associated to the
points $v_{1}, \ldots , v_k \in \pp E$. In other words, $F$ fits
into an exact sequence
\begin{equation} \label{elmEF}
0 \to E \to F \to \bigoplus_{i=1}^{k} \C_{x_i} \to 0,
\end{equation}
for $k$ distinct points $x_1, \ldots, x_k$ such that $v_i$ lies over
$x_i$. For each $i$, let $E_i$ be the intermediate sheaf ($ E
\subset E_i \subset F$) defined by the elementary
transformation associated to the $v_i$:
\[
0 \to E \to E_i \to \cc_{x_i} \to 0.
\]
By the Terracini Lemma and Lemma \ref{tangent}, we see that the
dimension of the embedded tangent space of $\pp E \subset
\pp$ at $v_i$ is given by the dimension of the linear span of the union of
\begin{equation}
\pp \Ker \left( H^{1}(X, \Sym^{2}E) \to H^{1}(X, \Sym^{2}E_i )
\right) \label{interker} \end{equation}
for $1 \le i \le k$. Since
$F$ is precisely the intersection of all the $E_i$ inside $E$, the
linear span of the spaces (\ref{interker}) coincides with
\[
\pp \Ker \left( H^{1}(X, \Sym^{2}E) \to H^{1}(X, \Sym^{2} F )
\right).
\]
Thus now it remains to show
\[
\dim \Ker \left( H^{1}(X, \Sym^{2}E) \to H^{1}(X, \Sym^{2}F )
\right) \ = \ \min \{ k(n+1), (n+1)e + \frac{1}{2}n(n+1)(g-1) \}.
\]
From the vanishing of $H^{0}(X, \Sym^{2}E)$, one checks that the
left hand side equals $k(n+1) - h^0(X, \Sym^2 F)$. Since $F$ is
obtained from a general elementary transformation of a general
bundle $E$, it is a general element of $\cU(n, -e + k)$. Hence the
required equality follows from  the result on $ \dim H^0(X, \Sym^2
F)$ which will be discussed in Lemma
\ref{vanishing} and Corollary \ref{dimension} in \S5.2. \qed

\begin{remark} (Hirschowitz bound for $\Gpn$-bundles)

The statement of Theorem \ref{thm1} is easily generalized to principal $\Gpn$-bundles. Recall
that the group $\Gpn$ of \emph{conformally symplectic
transformations} is the image of the multiplication map $\Spn \times
\cc^{*} \to \mathrm{\GL}_{2n}$. A principal $\Gpn$-bundle corresponds
to a vector bundle of rank $2n$ carrying a symplectic form with
values in a line bundle $L$ which may be different from $\cO_X$;
equivalently, admitting an antisymmetric isomorphism $W
\xrightarrow{\sim} W^{*} \otimes L$. Such $W$ has determinant $L^n$
and hence $\deg W = n \deg(L)$. (See Biswas--Gomez
\cite{BG} for more details on $\Gpn$-bundles.)

If $E \subset W$ is a Lagrangian subbundle then, by \cite[Criterion 2.1]{Hit2}, we get an extension
\[ 0 \to E \to W \to \Homom(E, L) \to 0 \]
with class $\delta(W) \in H^{1}(X, L^{-1} \otimes \Sym^{2} E)$, and
conversely. Arguing as
above, one can show that $W$ admits a Lagrangian subbundle of degree
at least $- \left\lceil \frac{1}{2}n \left( g - 1 - \deg(L) \right) \right\rceil$.
\end{remark}

\section{Geometry of the strata in rank four}

In this section, we prove Theorem \ref{mainthm} on the geometry of
the strata on $\cM_{4}$. For the case of genus 2, it has already
been proven in \cite{Hit3}. {\bf Throughout this section,
we assume $ \bf g \ge 3$}.

\subsection{Symplectic extensions of rank four}

First we study more details on symplectic extensions of rank four.
The goal is to prove the converse of the statement in Theorem
\ref{seginvsec} in the case of rank four. As was mentioned in Remark
\ref{converse}(2), we need to know how often diagrams of the form
(\ref{exceptional}) appear.

For $e \in \{ 1, 2, \ldots, g-1 \}$, let $E \to X$ be a vector
bundle of rank two and degree $-e$. Assume that $E$ is general, so $g-1 \le s_1(E) \le g$ as was remarked in (\ref{s_1}).
Equivalently, assume that any line subbundle of $E$ has degree at most
$\frac{1}{2}( -e - g + 1 )$.

\begin{lemma} \label{transverse}
Let $E$ be as above and consider an extension $0 \to E \to W \to
E^{*} \to 0$ (not necessarily symplectic). Let $F$ be a rank two
subbundle of $W$ such that the intersection of $E$ and $F$ is
generically of rank one. Then
\begin{enumerate}
\item $\deg F \leq -(g-1)$.
\item If $\deg E = \deg F = -(g-1)$, then the intersection of $E$ and $F$
is a line subbundle of  $W$ of degree $-(g-1)$.
\end{enumerate}
\end{lemma}
\begin{proof}
(1) Let $L$ be the line subbundle of $W$ associated to the intersection of $E$ and $F$. By the assumptions, we have a diagram
\begin{equation*}
\begin{CD}
0 @>>> E @>>> W @>>> E^{*} @>>> 0 \\
@. @AAA  @AAA  @AAA  @.\\
0 @>>> L @>>> F @>>> M @>>> 0
\end{CD}
\end{equation*}
where $M$ is a subsheaf of $E^*$. Since $E$ is general,
\[ \deg(L) \leq \frac{1}{2} ( -e - g + 1 ) \quad \hbox{and} \quad \deg(M) \leq \frac{1}{2} ( e - g + 1 ). \]
Hence $\deg(F) = \deg(L) + \deg(M) \le -(g-1)$.\\

(2) By the above inequalities, the condition $\deg E = \deg F =
-(g-1)$ implies that $\deg L = -(g-1)$ and $\deg M = 0$. If the quotient
sheaf $E^* /M$ had nonzero torsion, then $E^*$ would admit a quotient
line bundle of degree $<g-1$, contradicting the generality of
$E$. Hence $M$ must be a line subbundle of $E^*$, which shows that
the intersection of $E$ and $F$ coincides with $L$.
\end{proof}

The following is an immediate consequence of Lemma \ref{transverse} (1).

\begin{corollary}  \label{gensurg}
Let $E$ be a rank two bundle of degree $-e$ with $g-1 \le s_1(E) \le
g$. Consider a symplectic extension $0 \to E \to W \to E^{*} \to 0$
corresponding to a general point $\delta(W)$ of $\pp H^1(X, \Sym^2
E)$.
\begin{enumerate}
\item If $1 \le e \le g-2$, then any Lagrangian subbundle of $W$ other than $E$ itself, of
degree $\geq -e$, comes from an elementary transformation of $E^*$.
\item If $e = g-1$, then any Lagrangian subbundle of $W$ of degree $>
-e$ comes from an elementary transformation of $E^*$. \qed
\end{enumerate}
\end{corollary}

From this, we get a nice geometric criterion on lifting which
improves Theorem \ref{seginvsec} in the case of rank 4. Recall the
embedding criteria on $\psi$ of Lemma \ref{psiemb} (3), confirming
that for $g \geq 3$ and for a general bundle $E \in \cU(2, -e)$ with $e>0$,
the map $\psi \colon \pp E \to \pp \hoe$ is an
embedding. Hence in this case, we get a surface $\psi(\pp E) \cong \pp E$ inside $\pp H^1(X, \Sym^2 E)$.

\begin{theorem} \label{segsec}
Assume $g \ge 3$ and $1 \le e \le g-1$. Let $E$ be a general bundle in
$\cU(2,-e)$, and consider a nontrivial symplectic extension $W$
of $E^*$ by $E$. For each $k$ with $1 \le k \leq 2e-1$, the
following conditions are equivalent:
\begin{enumerate}
\item $W$ admits an isotropic lifting of an elementary
transformation $F$ of $E^*$ with $\deg F \ge e-k$.
\item $\delta(W) \in \Sec^{k}\pp E$ in $\pp
H^1(X, \Sym^2 E)$.
\item $\sLag(W) \leq (n+1)(k-e)$.
\end{enumerate}
\end{theorem}
\begin{proof} The implications (1) $\Rightarrow$ (2) and (2) $\Rightarrow$ (3) were already
shown in Lemma \ref{genlift} (1) and Theorem \ref{seginvsec}
respectively. The implication (3) $\Rightarrow$ (1) can be readily
seen as follows.

The condition $\sLag(W) \le (n+1)(k-e)$ implies that $W$ admits a
Lagrangian subbundle $F$ of degree at least $e-k$. By Corollary
\ref{gensurg}, if $\deg E = -e < e-k \le \deg F$, then $F$ comes from an elementary
transformation of $E^*$.
\end{proof}

\subsection{Stratification on $\cM_4$}

For any positive integer $e$, let $\cU(2, -e)^s$ be the moduli space
of stable bundles over $X$ of rank 2 and degree $-e$. According to
Narasimhan and Ramanan \cite[Proposition 2.4]{NR1975}, there exist a
finite \'etale cover $\pi_{e} \colon \tU_e \to \cU(2, -e)^s$ and a
bundle $\cE_e \to  \tU_e \times X$ with the property that $\cE_e|_{\{ E \} \times X}
\cong \pi_e(E)$ for all $E \in \tU_e$ (for
odd $e$, we can take $\pi_e$ to be the identity map since
$\cU(2,-e)$ is a fine moduli space).

Now by Riemann-Roch and semistability, for each $E \in \cU(2,
-e)^s$, we have
\[
\dim H^1(X, \Sym^{2}E) = 3e +3(g-1).
\]
Therefore, the sheaf $R^{1} {p}_{*} \Sym^2 (\cE_e)$ is locally free
of rank $3(e + g - 1)$ on $\tU_e$. Consider its projectivization
$\mu \colon \pp_{e} \to \tU_e$.

Now we have a diagram
\begin{displaymath}
\label{universal} \xymatrix{
\pp_{e}  \ar[dd]_{\mu} & \pp_{e} \times X \ar[d]^{\mu \times \Iden_X} \ar[l]_{r} & \\
& \tU_{e} \times X \ar[dl]_{p} \ar[dr]^{q} & \\
\tU_{e} & & X}
\end{displaymath}
We write $r : \pp_e \times X \to \pp_e$ for the projection. By Lange
\cite[Corollary 4.5]{Lan}, there is an exact sequence of vector
bundles
\[ 0 \to (\mu \times \Iden_{X})^{*}\cE_e \otimes r^{*}O_{\pp_e}(1)
\to \cW_{e} \to (\mu \times \Iden_{X})^{*}\cE_e^{*} \to 0 \] over
$\pp_{e} \times X$, with the property that for $\delta \in \pp_e$ with
$\mu(\delta) = E$, the restriction of $\cW_{e}$ to $\{ \delta \} \times X$ is isomorphic to the
symplectic extension of $E^*$ by $E$ defined by $\delta \in \pp H^1(X,
\Sym^{2}E)$.

There arises a basic question regarding these extension spaces: Is a
bundle $W$ corresponding to a general point $\delta \in \pp H^1(X,
\Sym^{2}E)$ stable, if $E$ is taken to be general? (The same
question for the vector bundles was answered affirmatively by
Brambila-Paz and Lange \cite{BPL}, and Russo and Teixidor i Bigas
\cite{RT}.) The machinery obtained in the previous subsection
enables us to answer this question.

\begin{lemma} \label{semistability}
Let $E \in \cU(2,-e)$, $1 \le e \le g-1$, be a general bundle such
that $g-1 \le s_{1}(E) \le g$. Then a general point of $\pp H^1(X,
\Sym^2 E)$ corresponds to a stable symplectic bundle.
\end{lemma}
\begin{proof}
By Theorem \ref{thm1}, a general symplectic bundle $W$ in $\cM_4$
satisfies $\sLag(W) = 3(g-1)$. This shows the above statement for
$e=g-1$.

Now assume $1 \le e<g-1$. Let $W$ be a symplectic bundle corresponding to a general point of
$\pp H^1(X, \Sym^2 E)$. Consider any line subbundle $L$ of $W$. If $L$ is contained in the subbundle $E$, then $\deg(L) < 0$ by the stability of $E$. Otherwise $L$ would yield an invertible
subsheaf of $E^*$ via the composition $L \to W \to E^*$. From the
condition on $s_1(E) = s_1(E^*)$, we have
\[ \deg L \le \frac{e - g + 1 }{2} < 0 \]
by generality of $E$. Therefore, any line subbundle
of $W$ has negative degree. Since symplectic bundles are self-dual,
the same holds for subbundles of rank three.

Finally we consider the Lagrangian subbundles of $W$. By Theorem
\ref{segsec}, $\sLag(W) \le 0$ if and only if $\delta(W) \in \Sec^e
\pp E$. But
\[ \dim \Sec^e \pp E \ \le \ 2e + (e-1) = 3e - 1, \]
while $\dim \pp H^1(X, \Sym^{2}E) = 3e + 3g-4$. Therefore, the
bundles with $\sLag(W) \le 0$ form a proper subset in $\pp H^1(X,
\Sym^{2}E)$.
\end{proof}

Now we go back to the diagram (\ref{universal}). There is a classifying map
$\gamma_{e} \colon \pp_{e} \dashrightarrow \cM_4$
induced by the bundle $\cW_e \to \pp_e \times X$. This map
$\gamma_e$ is defined over a nonempty dense subset of $\pp_e$.

Recall the
definition of the stratification on $\cM_4$ given by the invariant
$\sLag$: for each $e >0$, consider the subvarieties of $\cM_4$
defined by
\[
\cM_{4}^{e} := \{ W \in \cM_{4} \ : \ \sLag(W) \le 3e \}.
\]
Now we can show the following result which implies Theorem
\ref{mainthm}.
\begin{theorem} \label{main}
\begin{enumerate}
\item For each $e$ with $1 \le e \le g-1$, the map $\gamma_e$ is generically
finite and its image is dense in $\cM_4^{e}$.
\item For each $e$ with $1 \le e \le g-1$, the locus $\cM_4^{e}$ is
irreducible of dimension $7(g-1)+3e$.
\item For $e \le g-2$, a general point of $\cM_4^{e}$ corresponds
to a symplectic bundle which has a unique maximal Lagrangian
subbundle. In particular for odd $e \le g-2$, the locus $\cM_4^{e}$ is
birationally equivalent to the fibration $\pp_e$ over $\cU(2,-e)^s$.
\end{enumerate}
\end{theorem}
\begin{proof} \ First consider the case $e= g-1$. From the bound on $\sLag$
given in Theorem \ref{thm1}, we have $\cM_4 = \cM_4^{g-1}$. In
particular, $\cM_4^{g-1}$ is irreducible of dimension $10(g-1)$.
Moreover, the fact that a general symplectic bundle $W$ satisfies
$\sLag(W) = 3(g-1)$ implies that $\gamma_{g-1}: \pp_{g-1}
\dashrightarrow \cM_4$ is dominant. But $\dim \pp_{g-1} =
10(g-1)$, so $\gamma_{g-1}$ must be generically finite.

In general, it is clear that the image of $\gamma_e$ lands on the
stratum $\cM_4^{e}$. Since the source $\pp_e$ is irreducible, so is the
image of $\gamma_e$. Now we show that it is dense in
$\cM_4^e$. Any $W$ in $\cM_4^{e}$ is fitted into a symplectic extension of
$E^*$ by $E$ for some $E$ of degree $-e$ which might be
unstable. But every such $E$ is contained in an irreducible family
of bundles whose general member is a stable bundle in
$\cU(2, -e)$. This shows that $\cM_4^{e}$ is the closure
of the image of $\gamma_e$. In particular, $\cM_4^{e}$ is
irreducible for each $e$.

Now assume $1 \le e \le g-2$ and consider a point of
$\pp_e$, that is, a symplectic extension
\[ 0 \to E \to W \to E^* \to 0 \]
where $E \in \cU(2,-e)^s$ is general, so $s_1(E) \ge g-1$. Suppose
that $W$ admits a Lagrangian subbundle $F$ of degree $\ge -e$ other
than $E$. Then by Corollary \ref{gensurg} (1), $F$ is an elementary
transformation of $E^*$ which lifts to $W$. By Theorem \ref{segsec},
$\delta(W) \in \Sec^{2e} (\pp E)$. But
\[
\dim \Sec^{2e}(\pp E) \le 6e-1 \ < \ 3e+3g-4 = \dim \pp H^1(X, \Sym^2
E).
\]
Thus an extension represented by a general point of $\pp_{e}|_E$ contains no Lagrangian subbundle of degree $\ge
-e$ apart from $E$ itself. This implies that a general symplectic
bundle in $\cM_4^{e}$ has a unique maximal Lagrangian subbundle. In
other words, $W$ is represented only in the fibers of $\pp_e \to
\tU_e$ over the finite subset $\pi_e^{-1}(E)$ for the \'{e}tale
cover $\pi_e : \tU_e \to \cU(2, -e)^s$. Moreover, there is only one
extension class in $\pp H^1(X, \Sym^2 E)$ whose associated bundle is
isomorphic to $W$.

This shows that $\gamma_e$ is generically finite onto its image, of
degree equal to that of $\pi_e$. In particular, for odd $e \le
g-2$, $\gamma_e$ is generically injective since $\pi_e$ is the identity map, and so $\cM_4^{e}$ is
birationally equivalent to the fibration $\pp_e$ over $\cU(2,-e)^s$.
Also, for all $e$ we obtain $\dim \cM_4^{e} \ = \ \dim \pp_e \ = \ 7(g-1)+3e$.
\end{proof}


\subsection{Nonisotropic maximal subbundles} \label{difference}

By Serman \cite{Ser}, for $n>1$ the forgetful map taking a symplectic bundle to the equivalence class of the underlying vector bundle
gives an embedding of $\cM_{2n}$ in the moduli space $\mathcal{SU}(2n, \cO_X)$ of
semistable bundles of rank $2n$ with trivial determinant. Thus it is interesting to compare the two stratifications
on $\cM_{2n}$ given by $s_n$ and $\sLag$. We will focus on the rank four case.

Suppose $X$ has genus $g \geq 4$. Let $F_1$ and $F_2$ be a pair of
mutually nonisomorphic stable bundles of rank two and trivial
determinant. The direct sum $W := F_{1} \oplus F_2$ is a symplectic
bundle with $s_{2}(W) = 0$ and $\sLag(W) > 0$. We will now use Lemma
\ref{genlift} and Remark \ref{ordinarylifting} to illuminate this phenomenon.
\par
Let $E$ be a general bundle of rank two and determinant $\cO_{X}(-x)$
for some point $x \in X$. The constant function $1$ defines a global
rational section $\alpha$ of $\det(E)$ which has a simple pole at
$x$ and is elsewhere regular. Since $E$ has rank two, for \emph{any}
linearly independent $v, w \in E|_x$ we have, up to a constant,
\begin{equation}
\overline{\alpha} = \frac{v \wedge w}{z} = \frac{v \otimes w}{z} -
\frac{w \otimes v}{z}. \label{nonsymmp}
\end{equation}
Thus the cohomology class
\[ \left[ \frac{v \otimes w}{z} \right] = \left[ \frac{w \otimes v}{z} \right] \]
in $H^{1}(X, E \otimes E)$ defines a symplectic extension of $E^*$ by $E$.
\par
Write $p = \frac{v \otimes w}{z}$. By Lemma \ref{psiemb} (2), the map $\varphi \colon
\Delta \to \pp H^{1}(X, E \otimes E)$ is base point free. Since
$[p] = \varphi (v \otimes w)$, it is
a nontrivial cohomology class.
\par
By Lemma \ref{cohomlift}, the kernel $F$ of $p \colon E^{*} \to
\sPrin(E)$ lifts to a \emph{nonisotropic} subbundle of $W$. This
subbundle is an elementary transformation $0 \to F \to E^{*} \to
\cc_{x} \to 0$, so has trivial determinant (as it must, since
$h^{0}(X, \wedge^{2}F^{*}) > 0$).
\par
This behaviour can be explained geometrically as follows. Since
\begin{equation}
\left[ \frac{v \otimes w}{z} \right] = \left[ \frac{w \otimes v}{z}
\right] 
\label{equalclasses}
\end{equation}
for \emph{any} linearly independent $v$ and $w$, the map $\varphi
\colon \Delta|_{x} \to \pp H^{1}(X, E \otimes E)$ is not an
embedding. In fact it is a double cover
\[ \Delta|_{x} \cong \pp^{1} \times \pp^{1} \to \pp \Sym^{2}E|_{x} \cong \pp^2 \]
ramified over the plane conic $\pp E|_x$. (In particular, it
factorizes via $\pp H^{1}(X, \Sym^{2}E)$.) By (\ref{equalclasses}),
any extension class lying in this $\pp^2$ lies on (a 1-secant to)
the quadric bundle $\Delta$, but a general such class does not lie
on (a 1-secant to) the conic bundle $\pp E$. Therefore, by Lemma
\ref{genlift}, an elementary transformation of degree $\deg(E^{*}) -
1 = 0$ lifts to a nonisotropic subbundle of $W$, but $W$ has no
Lagrangian subbundle of degree zero.
\begin{remark}
In fact $W$ admits a pair of nonisotropic subbundles of trivial
determinant given by $\Ker( p )$ and $\Ker (^{t}p)$. It
is not hard to see that these are of the form $F$ and $F^{\perp}$.
Furthermore, the bundle $W$ splits as
the direct sum $F \oplus F^{\perp}$, by \cite[Theorem 2.3]{Hit1}.
\end{remark}

Building upon this idea, one can show the following:\\
\\
\textit{Suppose $g \geq 4$. For any $e$ and $f$ with $0 \leq f < e$
and $f + 2e - 1 \leq \frac{2(g-2)}{3}$, there exists a stable rank
four symplectic bundle $W$ with $s_{2}(W) = 4f$ and $\sLag(W) =
3e$. In other words, $W$ has a maximal subbundle of degree $-f$ which is nonisotropic, and a maximal \emph{Lagrangian} subbundle of degree $-e < -f$.}
\\
\\
The principle is as above: for certain $E$ of degree $-e$ and
special determinant, one can construct a class $\delta$ in $H^{1}(X,
\Sym^{2} E)$ lying on $\Sec^{e+f} \Delta \setminus \Sec^{2e-1}\pp
E$. By Lemma \ref{genlift}, the extension $0 \to E \to W \to E^{*}
\to 0$ defined by $\delta$ will have a nonisotropic subbundle of
rank two and degree $-f$, but no Lagrangian subbundle of degree
greater than $-e$. The details are tedious, so we omit the
calculation. We mention only that the condition $f + 2e - 1 \leq
\frac{2(g-2)}{3}$ comes from Hwang--Ramanan \cite[Proposition
3.2]{HR}, which guarantees the vanishing of $H^0(X, ad_{E}(D))$ for certain
effective divisors $D$.

\section{Proofs of two lemmas}

In this miscellaneous section, we provide proofs of two lemmas which
were used in the previous sections. Both statements may look rather obvious, but we could not find shorter proofs
than those given here.

\subsection{Irreducibility of spaces of symmetric principal parts}

Let $E$ be a vector bundle over $X$ of rank $n$. For a principal
part $p \in \Prin(E)$, we define the \textit{degree} of $p$ as the degree of the torsion sheaf $ \mathrm{Im} \left( p: E^* \to
\sPrin(E) \right)$. A principal part of degree
$d$ is called \textit{general} if it is supported at $d$ distinct
points of $X$. Clearly, the elementary transformation $\Ker
\left( p: E^* \to \sPrin(E) \right)$ of $E$ is general in the sense
of Definition \ref{genelm} if and only if $p$ is general.

In this subsection, we prove the following.
\begin{proposition}
Every symmetric principal part $p \in \Prin (E \otimes E)$ is a limit (in the analytic topology)
of a sequence of general symmetric principal parts of the same degree.
\end{proposition}
\begin{proof}
It suffices to consider the case where $p$ is supported at a
single point $x \in X$. In this case, we define the \textit{order} of $p$
to be the smallest integer $k$ such that $\textrm{Im}(p)$ is
contained in $E(kx)/E$. Let $z$ be a local coordinate at $x$ and
let $e_1, e_2, \ldots, e_n$ be a frame for $E$ near $x$. Then the symmetric principal part $p$ of order $k$ is locally expressed as
\[
p = \frac{\psi_k}{z^k} + \frac{\psi_{k-1}}{z^{k-1}} + \cdots +
\frac{\psi_1}{z}
\]
where $\psi_l$ is a symmetric tensor expressed by
\[
\psi_l = \sum_{1 \le i \le j \le n} a_{ij}^l (e_i \otimes e_j + e_j
\otimes e_i).
\]
Now we claim that after a suitable change of the local frame, every
symmetric principal part $p$ of degree $d$ and order $k$ can be
expressed by
\[
p = \sum_{i=1}^r p_{i} \quad \hbox{where} \quad p_{i} = \frac{e_i \otimes e_i}{z^{k_i}},
\]
and $k= k_1 \ge k_2 \ge \cdots \ge k_r$, and $\sum_{i=1}^r k_{i} = d$. It
is easy to see from this expression that $p$ is a limit of
a sequence of general symmetric principal parts of degree $d$: For each $i$,
choose $k_i$ distinct complex numbers $c_1, c_2, \ldots, c_{k_i}$
and consider the 1-parameter family of principal parts
\[
p_i(t) = \frac{e_i \otimes e_i}{(z-c_1t)(z-c_2t) \cdots
(z-c_{k_i}t)}.
\]
For $t\neq 0$, each $p_i(t)$ is a general principal part of degree $k_i$,
and $\sum_{i=1}^r p_i(0) = p$. Thus $p$ is the limit of a
sequence of general symmetric principal parts of degree
$\sum_{i=1}^r k_i = d$.

Now we prove the claim. Since $\psi_k$ is symmetric, we may assume
after a change of frame:
\[
\psi_k = e_1 \otimes e_1 + e_2 \otimes e_2 + \cdots + e_s \otimes
e_s
\]
for some $s \le n$. Now we make a change of local frame given by
\[
e_1' = e_1 + z \sum_{j=1}^n a_{1j}^k e_j.
\]
Then it is straightforward to check that $p$ is expressed as
\[
\frac{e_1' \otimes e_1'}{z^k} + \frac{\phi_k}{z^k} +
\frac{\phi_{k-1}}{z^{k-1}} + \cdots + \frac{\phi_1}{z},
\]
where $\phi_k, \phi_{k-1} \in \Sym^2<e_2, e_3, \cdots, e_n>$ \ and \
$\phi_{k-2}, \phi_{k-3}, \cdots, \phi_1 \in \Sym^2<e_1', e_2,
\cdots, e_n>$. For
\[
\phi_{k-2} = \sum_{j=1}^n b_{1j} (e_1' \otimes e_j + e_j \otimes
e_1') + \sum_{2 \le i \le j \le n} b_{ij} (e_i \otimes e_j + e_j
\otimes e_i),
\]
we make a change of local frame given by
\[
e_1'' = e_1' + z^2 (b_{11} e_1' +  \sum_{j=2}^n b_{1j} e_j).
\]
Then we can check that $p$ is expressed as
\[
\frac{e_1^{''} \otimes e_1^{''}}{z^k} + \frac{\phi_k'}{z^k} +
\frac{\phi_{k-1}'}{z^{k-1}} + \cdots + \frac{\phi_1'}{z},
\]
where $\phi_k', \phi_{k-1}', \phi_{k-2}' \in \Sym^2<e_2, e_3,
\cdots, e_n>$ \ and \ $\phi_{k-3}', \cdots, \phi_1' \in \Sym^2
<e_1'', e_2, \cdots, e_n>$.

Continuing this process, we finally get an expression for $p$ of the form
\[
\frac{e_1 \otimes e_1}{z^k} + q,
\]
where $q$ is a symmetric principal part expressed in $e_2, \ldots, e_n$.
Next we perform the same operation on $q$ and $e_2$. Continuing this process, finally we get an expression for
$p$ of the form
\[
p = \frac{e_1 \otimes e_1}{z^k} + \frac{e_2 \otimes e_2}{z^{k_2}} +
\cdots + \frac{e_r \otimes e_r}{z^{k_r}},
\]
as required.
\end{proof}

\begin{corollary}
For a fixed vector bundle $E$ over $X$ and for each $d>0$, the space
of symmetric principal parts in $\Prin(E \otimes E)$ of
degree $d$ is irreducible.
\end{corollary}

\subsection{A variant of Hirschowitz' Lemma}

Hirschowitz' lemma \cite[4.6]{Hir} ensures us that the tensor product of two general
bundles is nonspecial (see also Russo--Teixidor i Bigas \cite[Theorem 1.2]{RT}).
This is useful in many situations; for instance the proof \cite[p.12]{CH} of
Hirschowitz' bound on the Segre invariants. Here we will prove a variant of Hirschowitz' lemma.
\begin{lemma} \label{vanishing}
Let $F$ be a general stable bundle  of rank $n$ and degree $d$. If
$d \le \frac{1}{2} n(g-1)$, then $h^0(X, F \otimes F) = 0$.
\end{lemma}
\begin{proof} The proof will be completed in three steps.

\noindent \textit{Step 1}. \ Let $\mathbb{F}$ be a deformation of
$F$ given by a nonzero class $\delta \in H^{1}(X, \Endom(F))$. For a
given nonzero symmetric map $\alpha \colon F^* \to F$, we want to
know when there exists an extension $\tilde{\alpha} \colon
\mathbb{F}^{*} \to \mathbb{F}$ inducing the following commutative
diagram:
\begin{equation}
\begin{CD} \label{deformation}
0 @>>> F @>>>
\mathbb{F} @>>>  F  @>>> 0 \\
@. @AA{\alpha}A @AA\tilde{\alpha}A  @AA{\alpha}A  @.\\
0 @>>> F^* @>>> \mathbb{F}^* @>>>  F^*  @>>> 0.
\end{CD}
\end{equation}
We have induced maps
\begin{multline*} \alpha_{\flat} \colon H^{1}(X, \Endom(F^{*})) \to H^{1}(X, \Homom(F^{*}, F)) \\
\hbox{and} \quad \alpha^{\sharp} \colon H^{1}(X, \Endom(F)) \to H^{1}(X,
\Homom(F^{*}, F)). \end{multline*} Note that the class of the dual
deformation $\mathbb{F}^*$ is given by $-^{t}\delta \in H^{1}(X,
\Endom(F^{*}))$. By straightforward computation, we can check that the maps $\alpha$ on the outer terms in the diagram
(\ref{deformation}) extend to a map $\tilde{\alpha} \colon
\mathbb{F}^{*} \to \mathbb{F}$ if and only if $-\alpha_{\flat} \,^{t}\delta = \alpha^{\sharp}\delta$ in $H^{1}(X, \Homom(F^{*}, F))$.

Now since $\alpha$ is symmetric, we have
\[ -\alpha_{\flat} \,^{t}\delta = -\,(^{t}\alpha)_{\flat} \,^{t}\delta = -^{t}(\alpha^{\sharp}\delta). \]
Thus we obtain:
\begin{lemma}
\label{extcrit}
A nonzero symmetric map $\alpha : F^* \to F$ extends to $\tilde{\alpha}$ if and only if $\alpha^{\sharp}
\delta$ belongs to the subspace
$ H^{1}(X, \wedge^{2}F) \ \subseteq \ H^{1}(X, \Homom(F^{*}, F)). $ \qed
\end{lemma}
\textit{Step 2}. Rank 2 case: \ It suffices to consider the boundary
case when $\deg F = g-1$. Let $F$ be a general bundle in $\cU(2,
g-1)$. Since $\wedge^2 F = \det F$ is general in  $\mathrm{Pic}^{g-1} (X)$, we have
$H^0(X, \wedge^2 F) = 0$. Thus it suffices to show the vanishing
of $H^0(X, \Sym^2 F)$.

First we show that there is no nonzero map $F^* \to F$ whose image
is of rank 1: Since $s_1(F) = g-1$,
every line subbundle of $F$ has degree $\le 0$. Furthermore, by \cite[Corollary 3.2]{LN}), $F$ has
only finitely many maximal line subbundles (of degree 0). Thus if
there were a nonzero map $\alpha \colon F^* \to F$ with image $M$
of rank 1, then $\deg M = 0$ and $F$ should have both $M$ and
$M^*$ as its maximal subbundles. This implies either $M \cong M^*$
or there is a sequence
\[
0 \to M \oplus M^* \to F \to \tau \to 0
\]
for some torsion sheaf $\tau$ of degree $g-1$. By dimension
counting, one can check that neither of these conditions is
satisfied by a general $F \in \cU(2,g-1)$.

Next we show that if $\Sym^2 F$ has a nonzero section $\alpha$, then
it does not extend to every deformation of $F$. We have seen that
$\alpha : F^* \to F$ must be generically surjective. Hence the induced
cohomology map
\[
\alpha^{\sharp} \colon H^{1}(X, \Endom(F)) \to H^{1}(X, \Homom(F^{*}, F))
\]
is surjective. By the assumptions $h^0(X, \Sym^2 F) >0$ and
$\chi(\Sym^2 F)=0$, we see that $h^1(X, \Sym^2 F) >0$. This implies that
the image of $\alpha^{\sharp}$ is not contained in $H^1(X, \wedge^2 F)$. By
Lemma \ref{extcrit}, $\alpha$ does not extend to $\tilde{\alpha}$ for some deformation $\mathbb{F}$
of $F$. Therefore, $H^0(X, F \otimes F)$ vanishes for a general $F \in \cU(2,g-1)$.
\\
\textit{Step 3}. Induction for higher rank cases: \ Now we consider
bundles of rank $n \ge 3$. Firstly, suppose $n$ is even. By
semicontinuity, it suffices to find a bundle $F_0 \in \cU(n,
\frac{1}{2}n(g-1))$ such that $H^0(X, F_0 \otimes F_0) = 0$. We let
$F_0 = F_1 \oplus F_2$ where $F_1$ is general in $\cU(2, g-1)$ and
$F_2$ is general in $\cU(n-2, \frac{1}{2} (n-2)(g-1))$, so that
$F_0$ is a polystable bundle of rank $n$ and degree $\frac{1}{2} n
(g-1)$. From the generality condition and the induction hypothesis,
$F_1 \otimes F_1$ and $F_2 \otimes F_2$ have no nonzero sections.
Also $h^0(X, F_1 \otimes F_2) = 0$ by Hirschowitz'
lemma \cite[4.6]{Hir}, \cite[Theorem 1.2]{RT}. Therefore, $F_0
\otimes F_0$ has no nonzero sections.

Next suppose both $n$ and $g$ are odd. Let $F_0 = F_1 \oplus F_2$
for a general line bundle $F_1$ of degree $\frac{g-1}{2}$ and a
general $F_2 \subset \cU(n-1, \frac{1}{2} (n-1)(g-1))$. By the same
argument as above, we see that $h^0(X, F_0 \otimes F_0) = 0$.

Finally suppose that $n$ is odd and $g$ is even. In this case, it
suffices to find a bundle $F_0 \in \cU(n, \frac{1}{2}(n(g-1) -1))$
such that $H^0(X, F_0 \otimes F_0) = 0$. Let $F_0$ be a general
extension
\[
0 \to F_1 \to F_0 \to F_2 \to 0
\]
for a general line bundle $F_1$ of degree $\frac{g-2}{2}$ and a
general $F_2 \in \cU(n-1, \frac{1}{2}(n-1)(g-1))$. From the stability of $F_1$ and $F_2$,
it is easy to check that $F_0$ is also stable.
Again by the same
argument as above, we see that $h^0(X, F_0 \otimes F_0) = 0$.
\end{proof}
\begin{corollary} \label{dimension}
If $F$ is a general bundle of rank $n$ and degree $d \geq
\frac{1}{2} n(g-1)$, then
\[
h^0(X, \Sym^2 F) = (n+1)d - \frac{1}{2}n(n+1)(g-1).
\]
\end{corollary}
\begin{proof}
For a theta characteristic $\kappa$ of $X$, we have
\begin{eqnarray*}
h^1(X, \Sym^2 F) &\cong& h^0(X, K_X \otimes \Sym^2 F^*)\\
& =& h^0(X, \Sym^2(\kappa \otimes F^*)).
\end{eqnarray*}
Since $\deg (\kappa \otimes F^*) = n(g-1) -d \le \frac{1}{2}n(g-1)$,
we can apply Lemma \ref{vanishing} to get the vanishing of $H^1(X,
\Sym^2 F)$. By Riemann-Roch, we get the wanted result on $h^0(X, \Sym^2 F)$.
\end{proof}

\begin{remark}
In the case when either $n$ is even or $g$ is odd, the assignment $F
\mapsto F \otimes F$ defines a morphism
\[ \cU (n, \frac{1}{2}n(g-1) ) \to \cU(n^{2}, n^{2}(g-1)), \]
since a tensor product of semistable bundles is semistable. The
target moduli space has a generalized theta divisor, whose support
consists of semistable bundles with nonzero sections. Lemma
\ref{vanishing} implies that the image of this morphism is not contained in the
generalized theta divisor.
\end{remark}

\section*{Acknowledgements}

The first named author was supported by the National Research Foundation of Korea (NRF) grant funded by the Korea government(MEST) (No.\ 2010-0001194).

The second author thanks Konkuk University for generous hospitality and financial support on two occasions.

\vspace{1cm}

\noindent Department of Mathematics, Konkuk University, 1 Hwayang-dong, Gwangjin-Gu, Seoul 143-701, Korea.\\
Email: \texttt{ischoe@konkuk.ac.kr}\\
\\
H\o gskolen i Vestfold, Boks 2243, N-3103 T\o nsberg, Norway.\\
Email: \texttt{george.h.hitching@hive.no}

\end{document}